\documentclass[fontsize=11pt]{amsart}

\usepackage{amsmath, amsthm, amssymb, mathtools, braket, commath, xfrac, gensymb, mathrsfs, physics, centernot}
\usepackage{verbatim}
\usepackage{stmaryrd}
\usepackage[margin=1in]{geometry}
\usepackage{scrextend}
\usepackage{tikz-cd}
\usepackage{enumitem}
\usepackage{titlesec}
\usepackage[pdfencoding=auto]{hyperref}
\hypersetup{
  colorlinks= true,
  linkcolor=[RGB]{0,89,42},
  urlcolor=[RGB]{0,89,42},
  citecolor=[RGB]{0,89,42}
}
\usetikzlibrary{calc, intersections}

\usepackage{colonequals}
\usepackage{bm}

\titlelabel{\thetitle.\hspace{.2cm}}

\newcommand{\Z}{\mathbb{Z}}

\newcommand{\Q}{\mathbb{Q}}

\newcommand{\cE}{\mathcal{E}}

\newcommand{\cA}{\mathcal{A}}
\newcommand{\M}{\mathcal{M}}

\newcommand{\F}{\mathcal{F}}
\newcommand{\B}{\mathcal{B}}
\newcommand{\cC}{\mathcal{C}}
\newcommand{\cD}{\mathcal{D}}
\newcommand{\K}{\mathcal{K}}

\newcommand{\fX}{\mathfrak{X}}

\newcommand{\w}{\omega}

\renewcommand{\O}{\mathcal{O}}

\renewcommand{\L}{\mathbb{L}}
\renewcommand{\P}{\mathbb{P}}

\newcommand{\vir}{\mathrm{vir}}
\newcommand{\BL}{\mathrm{BL}}

\newcommand{\cS}{\mathcal{S}}

\DeclareMathOperator{\Aut}{Aut}

\DeclareMathOperator{\Sym}{Sym}

\DeclareMathOperator{\Pic}{Pic}

\DeclareMathOperator{\Jac}{Jac}

\DeclareMathOperator{\Bl}{Bl}

\DeclareMathOperator{\genus}{genus}
\DeclareMathOperator{\Def}{Def}

\DeclareMathOperator{\gon}{gon}

\allowdisplaybreaks

\theoremstyle{plain} \newtheorem{thm}{Theorem}[section]
\newtheorem{lem}[thm]{Lemma}
\newtheorem{prop}[thm]{Proposition}
\newtheorem{cor}[thm]{Corollary}

\theoremstyle{definition}

\newtheorem{remark}[thm]{Remark}

\newtheorem{definition}[thm]{Definition}
\newtheorem{set-up}[thm]{Set-Up}
\newtheorem{problem}[thm]{Problem}

\newtheorem*{acknowledgements*}{Acknowledgements}

\newcounter{tmp}

\begin{document}

\title{Nodal Elliptic Curves on K3 Surfaces}

\author{Nathan Chen, Fran\c{c}ois Greer and Ruijie Yang}

\maketitle

\thispagestyle{empty}

\begin{abstract}

Let $(X,L)$ be a general primitively polarized K3 surface with $c_1(L)^2 = 2g-2$ for some integer $g \geq 2$. The Severi variety $V^{L,\delta} \subset \abs{L}$ is defined to be the locus of reduced and irreducible curves in $\abs{L}$ with exactly $\delta$ nodes and no other singularities. When $\delta=g$, any curve $C \in V^{L,g}$ is a rational curve; in fact, Chen \cite{Chen02} has shown that all rational curves in $\abs{L}$ are nodal, and the number of such rational curves is given by the Yau-Zaslow formula \cite{YZ96}.

In this paper, we consider the next case where $\delta = g-1$ and the Severi variety $V^{L,g-1}$ parametrizing nodal elliptic curves is of dimension 1. Let $\overline{V}^{L,g-1} \subset \abs{L}$ denote the Zariski closure. For a reduced curve $C$, we define the \textit{geometric genus} of $C$ to be the sum of the genera of the irreducible components of the normalization. We prove that the geometric genus of the closure $\overline{V}^{L,g-1} \subset \abs{L}$ is bounded from below by $O(e^{C\sqrt{g}})$.

\end{abstract}

\section*{Introduction}

Let $(X,L)$ be a general primitively polarized K3 surface with $c_1(L)^2 = 2g-2$ for some integer $g \geq 2$.  The linear system $\abs{L} \cong \P^{g}$ is the base of a family of curves $C$ on $X$ with arithmetic genus $g$. For $0\leq \delta\leq g$, we define the \textit{Severi variety} $V^{L,\delta}\subset \abs{L}$ to be the locus of (irreducible) curves with $\delta$ nodes, which is codimension $\delta$ in $\abs{L}$. 

When $\delta=g$, any curve $C \in V^{L,g}$ is a rational curve; in fact, Chen \cite{Chen02} has shown that all rational curves in $\abs{L}$ are nodal. The number $N_g$ of such rational curves\footnote{Strictly speaking, $N_0=1$ (resp. $N_1=24$) are interpreted as counts of rational curves in a class of self-intersection $-2$ (resp. 0), which is not a polarization class.} is given by the Yau-Zaslow formula \cite{YZ96}
\[ \sum_{g=0}^\infty N_g q^g = \frac{q}{\Delta(q)} = 1+24q + 324 q^2 + \ldots \]
in terms of modular forms.

In this note, we focus on the next case when $\delta=g-1$. By the \textit{compactified Severi curve} we mean the closure
\[ \overline{V}^{L,g-1} \subset \abs{L}. \]
One can ask about the geometry of this curve. For example, what is the asymptotic behavior of the geometric or arithmetic genus? Is it irreducible? The best known result in this direction is that $\overline{V}^{L, \delta}$ is irreducible for general K3 surfaces when $\delta \leq g-4$ \cite{BLC21}. See also \cite{CD19} for irreducibility when $g \geq 11$ and $\delta\leq \frac{g+3}{4}$.

For a reduced curve $C$, we define the \textit{total geometric genus} of $C$ to be the sum of the genera of the components of the normalization. Our main result gives a lower bound for the total geometric genus of $\overline{V}^{L,g-1}$.

\begingroup
\setcounter{tmp}{\value{thm}}
\setcounter{thm}{0} 
\renewcommand\thethm{\Alph{thm}}
\begin{thm}\label{thm:Main}
On a very general primitively polarized K3 surface $(X, L)$ of genus $g$, the total geometric genus of the compactified Severi curve is bounded from below:
\[ \genus \left( \overline{V}^{L,g-1} \right) \geq O(e^{C\sqrt{g}}), \]
for some constant $C > 0$. In particular, it goes to $\infty$ as $g$ does.
\end{thm}
\endgroup
In \cite{EL18}, it is shown that $\overline{V}^{L,\delta}$ does not contain any pencils
\[ \P^{1} \subset \abs{L} \]
for $\delta \gtrsim g - \sqrt{\frac{g}{2}}$. In the same paper \cite[Remark 2.4]{EL18}, the authors ask whether $\overline{V}^{L,\delta}$ exhibits hyperbolic properties when $\delta$ is large. Our result provides positive evidence for this question in this first case. However, it is possible that our space contains many components with small geometric genus.

We now give a sketch of the proof of Theorem~\ref{thm:Main}. First we observe that on a very general K3 surface $(X, L)$, the compactified Severi curve is birational to the main component\footnote{Here we mean union of all the components of dimension 1.} of the Kontsevich moduli space of genus 1 stable maps to $X$ (see Lemma~\ref{severikontsevich}). After specializing to a very general hyperelliptic K3 surface $(X_{0}, L_{0})$, the total geometric genus of $\overline{V}^{L,g-1}$ is bounded from below by the total geometric genus of the reduced components of the flat limit $\overline{\M}_{1}^{\lim}(X_{0}, L_{0})$. The class of this flat limit is related to the reduced Gromov-Witten virtual class $[\overline{\M}_{1}(X_{0}, L_{0})]^{\vir}$ by deformation invariance. We then identify a substack $\overline{\F} \subset \overline{\M}_{1}(X_{0}, L_{0})$ whose components dominate a curve $\overline{\Omega}$ of positive genus. Adapting the virtual counting techniques in \cite{BL00}, we compute the degree of the cover of $\overline{\F}$ over $\overline{\Omega}$ to obtain the desired lower bound for the geometric genus of the compactified Severi curve.

\begin{acknowledgements*}
We would like to thank Jim Bryan, Rob Lazarsfeld, and David Stapleton for reading earlier versions of the paper and giving valuable feedback. Furthermore, we would like to thank Huai-Liang Chang, Joe Harris, and Aleksey Zinger for helpful discussions. We thank the anonymous referee for their valuable comments and for identifying a crucial gap in the original argument.
\end{acknowledgements*}

%
%

\section{Geometry of Severi varieties}\label{sec:background severi}

Let $X$ be a complex K3 surface, i.e., a projective surface over $\mathbb{C}$ with $H^{1}(X, \O_{X}) = 0$ and $\w_{X} \cong \O_{X}$, and let $L \in \Pic(X)$ be an ample line bundle with $c_1(L)^2 = 2g - 2$ where $g\geq 2$. By Riemann-Roch, we have $h^{0}(X, L) = g + 1$ and every curve $C \in \abs{L}$ has arithmetic genus $p_{a}(C) = g$.

\begin{definition}
Let $\K_{g}$ denote the irreducible 19-dimensional moduli stack of primitively polarized K3 surfaces $(X, L)$ of genus $g \geq 2$.
\end{definition}

\subsection{Severi varieties.}

For a given integer $\delta \geq 0$, the \textit{Severi variety} $V^{L, \delta}$ is the subset of $\abs{L}$ consisting of integral curves with $\delta$ nodes and no other singularities, so their normalization has genus $g - \delta$. Recall the following facts about Severi varieties:

\begin{prop}[Example 1.3 of \cite{CS97}]\label{prop:SeveriSmooth}
If non-empty, the Severi variety $V^{L, \delta}$ is smooth and has pure dimension $g-\delta$.
\end{prop}

\begin{prop}[Proposition 4 of \cite{CD19}]
For the general $(X, L) \in \K_{g}$, the Severi variety $V^{L, \delta}$ is non-empty for all non-negative integers $\delta \leq g$.
\end{prop}

In this paper, we focus on the case where $\delta = g - 1$. These results imply that for general $(X, L) \in \K_{g}$, $V^{L, \delta}$ is a smooth curve.

\subsection{Severi varieties via specialization.} 

We will study the Severi variety of a general K3 surface by specializing to a hyperelliptic K3 surface. The case $g=2$ is special because the general such K3 surface $(X,L)$ is already hyperelliptic. In this case, the Severi curve $\overline{V}^{L,1}$ is the dual curve of the sextic branch curve in $\P^2$, so it has geometric genus 10. From now on, let us assume that $g \geq 3$.

If $(X,L)$ is a K3 surface satisfying $c_1(L)^2 = 2g-2 \geq 4$, and $L \in \Pic(X)$ has no fixed components, then:

\begin{thm}[{see \cite[\S 4]{Saint-Donat74}}]
The linear system $\abs{L}$ has no base points, and hence defines a morphism
\[ \phi_{L} \colon X \rightarrow \P^{g}. \]
There are two cases.
\begin{enumerate}[label=\arabic*.]
    \item If $\abs{L}$ contains a non-hyperelliptic curve of genus $g$, then $\phi_{L}$ is birational onto a surface of degree $2g - 2$ with only isolated rational double points.
    \item If $\abs{L}$ contains a hyperelliptic curve, then $\phi_{L}$ is a generically 2-to-1 mapping of $X$ onto a surface $\mathbb{F}$ of degree $g-1$.
\end{enumerate}
\end{thm}

We refer to the second case as a \textit{hyperelliptic} K3. Throughout the paper, $(X, L)$ will denote a very general K3 surface and $(X_0,L_0)$ a hyperelliptic K3 surface. Recall the following result due to Reid:

\begin{thm}[\cite{Reid76}]
For any hyperelliptic polarized K3 surface $(X_0,L_0)$ with $g\geq 3$, the image $\phi_{L_0}(X_0)$ is a Hirzebruch surface $\mathbb{F}_{n}$ ($n = 0, 1, 2, 3, 4$), with ramification curve $R \in \abs{-2K_{\mathbb{F}_{n}}}$. Furthermore, $n=0$ occurs when $g$ is odd, and $n=1$ occurs when $g$ is even.
\end{thm}

Hyperelliptic K3 surfaces form a Noether-Lefschetz divisor in $\K_{g}$, so any primitively polarized $(X,L) \in \K_{g}$ can be specialized to a hyperelliptic $(X_0,L_0)$ covering $\mathbb{F}_0$ or $\mathbb{F}_1$. The Picard group of such a hyperelliptic K3 has rank $\geq 2$, and contains the primitive lattice
\[ \langle L_0,M_0\rangle = \begin{pmatrix} 2g-2 & 2 \\ 2 & 0 \end{pmatrix}. \]
To see this, there are two cases:
\begin{itemize}
\item If $g=2r+1$, it suffices to study the hyperelliptic K3 surfaces $X_0 \xrightarrow{\pi} \mathbb{F}_{0} \cong \P^{1} \times \P^{1}$ branched along a smooth curve $B\subset \mathbb{F}_0$ of bidegree $(4,4)$. The line bundle $L_0 = \pi^{\ast}\O(1, r)$ gives a polarization of genus $2r+1$, and along with $M_0=\pi^{\ast}\O(0,1)$ forms the sublattice above.  The moduli space of $(4,4)$-curves has dimension
\[ \dim |(4,4)| - \dim \Aut(\P^1\times \P^1)= 18. \]
\item If $g=2r$, it suffices to study the hyperelliptic K3 surfaces $X_0 \xrightarrow{\pi} \mathbb{F}_{1}$ branched along a smooth curve $B\in |4e+6f|$, where $e$ is the section with $e^2=-1$, and $f$ is the class of a fiber. The line bundle $L_0 = \pi^{\ast}\O(e+rf)$ gives a polarization of genus $2r$, and along with $M_0=\pi^{\ast} \O(f)$ forms the sublattice above. The moduli space of branch curves $B$ in $\mathbb{F}_1$ has dimension
\[ \dim |4e+6f| - \dim \Aut(\Bl_0\P^2) = 18. \]
\end{itemize}

For the rest of the paper, we will focus on the first case when the genus $g$ is odd; however, the argument is similar for the even genus case.

\begin{set-up}\label{setup:hyperelliptic K3}
Let $(X_0,L_0)$ be a hyperelliptic K3 surfaces of genus $g = 2r + 1$ which is a double cover of $\P^{1} \times \P^{1}$ branched along a smooth $(4,4)$-curve $B$, with $L_0 \colonequals \pi^{\ast}\O(1, r)$. 
\end{set-up}

\begin{center}
\includegraphics[width=10cm]{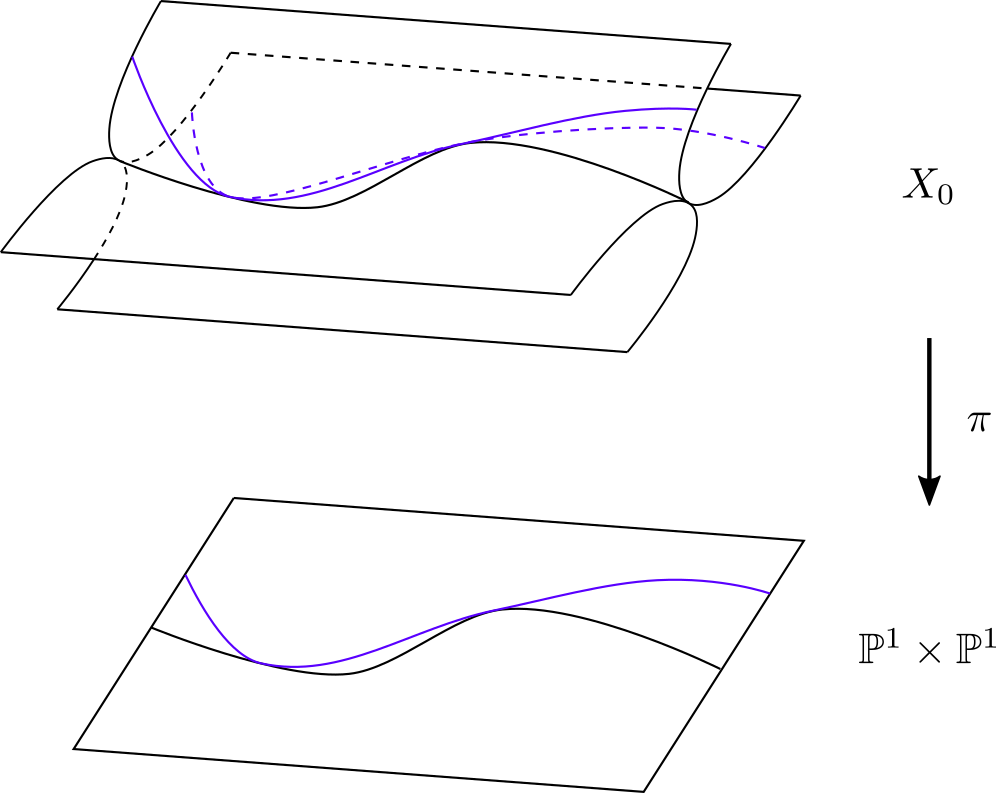}
\end{center}

\begin{remark}\label{tangentnodes}
Notice that divisors in $\P^{1} \times \P^{1}$ which are ordinarily tangent to the branch curve $B$ give rise to nodal curves on the hyperelliptic K3 surface. On the other hand, $H^{0}(\P^{1} \times \P^{1}, \O(1, r)) \cong H^{0}(X_0, \pi^{\ast}\O(1, r))$ so the image under $\pi$ of an integral nodal curve in $\abs{\pi^{\ast}\O(1, r)}$ must lie in $\abs{\O(1, r)}$. This image is a smooth $\P^{1}$, and hence has no nodes. Therefore, the $\delta$-nodal curves on $X_{0}$ must come from divisors in $\P^1\times \P^1$ which are $\delta$ times ordinarily tangent to $B$.
\end{remark}

%
%

\section{Genus one stable maps on a general K3 surface}

Let $X$ be a smooth projective variety and $\beta\in H_2(X,\Z)$. The Kontsevich moduli stack $\overline{\M}_g(X,\beta)$ of genus $g$ stable maps to $X$ in class $\beta$ is a proper Deligne-Mumford stack \cite{FP96} whose $\mathbb{C}$-points parametrize morphisms
\[ f \colon C \rightarrow X, \]
where $C$ is a connected nodal curve of arithmetic genus $g$, $f_{\ast}[C]=\beta$, and $\abs{\text{Aut}(f)}<\infty$. In a slight abuse of notation, we will often write $L$ instead of the Poincar\'{e} dual of $c_1(L)$.

For $X$ a very general K3 surface, we compare $\overline{\M}_{1}(X, L)$ with the compactified Severi curve $\overline{V}^{L, g-1}$. Let $\overline{\M}_1^\circ(X,L)$ denote the closed-open substack of stable maps with irreducible domain curve, which is a scheme.
The following result has appeared in \cite{Chen19}, but we would like to rephrase it in our notation.

\begin{lem}\label{lem:worstnodal}
On a very general primitively polarized K3 surface $(X, L)$, the generic element $[f]\in \overline{\M}_1^\circ(X,L)$ has image which is at worst nodal.
\end{lem}
\noindent By Lemma~\ref{lem:worstnodal}, we know that $\overline{\M}_{1}^{\circ}(X,L)$ is generically reduced because the Zariski tangent space to the moduli space at an immersed curve $[f] \in \overline{\M}_{1}^{\circ}(X,L)$ is 1-dimensional. By Proposition~\ref{prop:SeveriSmooth}, the Severi variety $V^{L,g-1}$ is reduced. It follows that:
\begin{cor}\label{severikontsevich}
Let $(X,L)$ be a very general K3 surface. There is a birational morphism
\[ \overline{\M}_1^{\circ}(X,L)  \to  \overline{V}^{L,g-1} \]
on each component, defined by taking the image of a stable map. In particular, they have the same geometric genus.
\end{cor}

Since $c_1(L)$ is primitive and irreducible, the only excluded components of $\overline{\M}_{1}(X, L)$ consist of stable maps with a contracted component of genus one.
\begin{definition}\label{def:ghost}
If a stable map $f:C\to X$ contracts a component $C_0\subset C$, then we call $C_0$ a \textit{ghost} or \textit{ghost component}.
\end{definition}

\noindent In the case of $\overline{\M}_1(X,L)$, a map with an elliptic ghost will have rational image. Each component of $\overline{\M}_1(X,L)$ consisting of maps with elliptic ghosts is of the form $R\times \overline{\M}_{1,1}$, where $R\in |L|$ is a rational curve.

\begin{lem}\label{lem:ghost}
For $(X,L)$ very general, the components of $\overline{\M}_{1}(X, L)$ consisting of maps with elliptic ghosts are all disjoint from $\overline{\M}_1^\circ(X,L)$.
\end{lem}

\begin{proof}

Since $\overline{\M}_{1}(X, L)$ is proper, any family of stable maps has a unique limit. Consider a family in $\overline{\M}_1^\circ(X,L)$ limiting to a stable map with rational image.  By \cite{Chen02}, the image must be nodal, so a flat limit may be obtained by allowing the domain to acquire a node. Such a stable map has no ghost component.
\end{proof}

\begin{remark}\label{remark:virtualclass}
The Kontsevich moduli stack $\overline{\M}_1(X,L)$ carries a perfect obstruction theory $E^\bullet$ coming from reduced Gromov-Witten theory (see $\S \ref{subsec:VFC}$ for more details) and an associated virtual fundamental class
\[ [\overline{\M}_1(X,L)]^{\vir}\in CH_{1}( \overline{\M}_1(X,L),\Q). \]
Throughout the paper, we will be working with reduced virtual classes because $X$ is a K3 surface. On $\overline{\M}_1^{\circ}(X,L)$, the virtual class agrees with the fundamental class, since it has dimension 1. To bound the geometric genus of this curve, let $\mathcal{X} \to \Delta$ be a hyperelliptic specialization such that $\mathcal{X}_t\cong X$ and $\mathcal{X}_0=X_0$ is a hyperelliptic K3 surface. Now consider the flat limit of $\overline{\M}_1^{\circ}(X,L)$ inside $\overline{\M}_1(X_0,L_0)$, which we denote by
\[ \overline{\M}^{\mathrm{lim}}_1(X_0,L_0) \subset \overline{\M}_1(X_0,L_0). \]
\noindent We will show that on the hyperelliptic K3, the class $[\overline{\M}_1(X_0,L_0)]^{\vir}$ decomposes into a sum of $[\overline{\M}^{\mathrm{lim}}_1(X_0,L_0)]$ together with an effective cycle class contribution from the ghost components (see \S\ref{subsec:definvariance}).
\end{remark}

The curve $\overline{\M}^{\mathrm{lim}}_1(X_0,L_0)$ may have nonreduced components, but we will show that it contains a collection of reduced components with large total geometric genus. The following lemma will then give us a lower bound for the total geometric genus of the compactified Severi curve.

\begin{lem}\label{lemma:semicontinuityGeomGenus}
Let $\mathcal V \to \Delta$ be a 1-parameter flat family of (possibly singular) curves with reduced general fiber. Decompose the central fiber into $V_{r} \cup V_{nr}$, where $V_{r}$ consists of all reduced components, and $V_{nr}$ consists of all nonreduced components. Then the general fiber $V_{t}$ has total geometric genus bounded from below by:
\[ \genus(\mathcal V_t) \geq \genus(V_r). \]
\end{lem}
\begin{proof}
After base change and normalization of the total space $\mathcal V$, we may assume that the general fiber $\mathcal V_t$ is a smooth curve (possibly disconnected), and the monodromy action on the connected components is trivial. Next, we apply stable reduction to each component separately. The algorithm in \cite{Knudsen} uses a sequence of prime order base changes, normalizations, and blow ups to produce a semi-stable model. Each of these operations can only raise the geometric genus of a given component of $V_r$, and all nonrational components will persist to the final stable model. The result now follows from lower semi-continuity of geometric genus in a family of stable curves.
\end{proof}

%
%

\section{Genus one stable maps on a hyperelliptic K3 surface}\label{sec:hyperstable}

In this section, we first recall the moduli stacks defined by Bryan and Leung in \cite[\S 4]{BL00} to deal with multiple covers of nodal fibers of elliptic K3 surfaces. We then define a large component of $\overline{\M}_{1}(X_0,L_0)$ on the hyperelliptic K3 surface $X_0$, and identify a divisor in that component with a certain product of moduli stacks coming from \cite{BL00}.

\subsection{Bryan-Leung spaces.}

The surface $Y$ under consideration in \cite{BL00} is an elliptically fibered K3 surface over $\P^1$ with a unique section $S$ and 24 singular nodal fibers $N_{1}, \ldots, N_{24}$. Let $F$ be the class of the fiber, so that
\[ F^{2} = 0, \quad F \cdot S = 1, \quad S^{2} = -2. \]

\begin{definition}
We say that a sequence $s = \{ s_{n} \}$ is \textit{admissible} if each $s_{n}$ is a positive integer and the index $n$ runs from $-j, -j + 1, \ldots, k$ for some integers $j, k \geq 0$. Write $\abs{s}$ for $\sum_{n} s_{n}$. 
\end{definition}

\begin{enumerate}[label=\roman*.]
    \item For $\vec{\bm{a}}=(a_1,\ldots,a_{24}) \in \Z_{\geq 0}^{24}$, let $\smash{\overline{\M}}\vphantom{M}^{\BL}_{\vec{\bm{a}}}$ be the moduli stack of stable, genus 0 maps to $Y$ with image $S + \sum_{i=1}^{24} a_{i} N_{i}$.
    \item Let $\smash{\overline{\M}}\vphantom{M}^{\BL}_{a}$ be the moduli space of stable, genus 0 maps to $Y$ with image $S + aN$ for any fixed nodal fiber $N$.
    \item Let $\Sigma(a)$ be a genus 0 nodal curve consisting of a linear chain of $2a+1$ smooth rational components $\Sigma_{-a}, \ldots, \Sigma_{a}$ with an additional smooth rational component $\Sigma_{\ast}$ meeting $\Sigma_{0}$ (so that $\Sigma_{n} \cap \Sigma_{m} = \emptyset$ unless $\abs{n-m} = 1$ and $\Sigma_{\ast} \cap \Sigma_{n} = \emptyset$ unless $n = 0$). Given an admissible sequence $s(a)=\{ s_{n}(a) \}$ with $\abs{s(a)} = a$, we define $\smash{\overline{\M}}\vphantom{M}^{\BL}_{s(a)}$ to be the moduli space of genus 0 stable maps with $\Sigma(a)$ as the target in the class
\[ \Sigma_{\ast} + \sum_{n=-a}^{a} s_n(a) \Sigma_n.\]
The admissible sequence has been extended by zero (if necessary) so that $n$ runs from $-a$ to $a$.
\end{enumerate}

\begin{center}
\includegraphics[width=14cm]{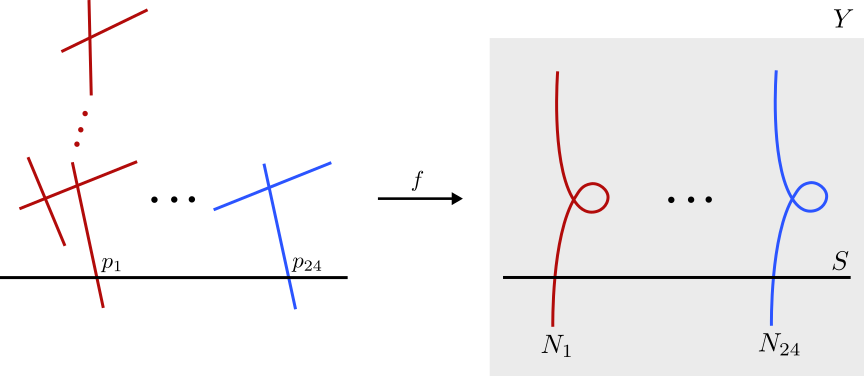}
\end{center}

With this notation in mind, Bryan and Leung prove the following result:

\begin{lem}[Lemma 5.3 of \cite{BL00}]\label{lem:admissible sequence}
For a fixed nodal fiber $N$, the moduli space $\smash{\overline{\M}}\vphantom{M}^{\BL}_{a}$ is a disjoint union $\coprod_{s(a)} \smash{\overline{\M}}\vphantom{M}^{\BL}_{s(a)}$ labeled by admissible sequences $s(a) = \{ s_n(a)\}$ with  $\abs{s(a)}=a$.
\end{lem}

For each $k>0$, notice that the images of stable maps in $\overline{\M}_0(Y,S+kF)$ do not contain any smooth elliptic fibers because the domain is a genus 0 curve.  The Kontsevich moduli stack decomposes as follows:
\begin{align*}
\overline{\M}_0(Y,S+kF) &\cong \coprod_{a_1+\dots+a_{24}=k} \smash{\overline{\M}}\vphantom{M}^{\BL}_{\vec{\bm{a}}} \\
&\cong \coprod_{a_1+\dots+a_{24}=k}\left( \prod_{i=1}^{24} \smash{\overline{\M}}\vphantom{M}^{\BL}_{a_i}\right) \\
&\cong \coprod_{a_1+\dots+a_{24}=k}\left( \prod_{i=1}^{24} \coprod_{s(a_{i})} \smash{\overline{\M}}\vphantom{M}^{\BL}_{s(a_{i})}\right).
\end{align*}
Here, $\smash{\overline{\M}}\vphantom{M}^{\BL}_{s(a_{i})}$ is identified with stable maps to $Y$ which factor through $\Sigma(a_i)$. The $n$-th term in the admissible sequence $s(a_i)$ indicates the degree of the stable map onto the component $\Sigma_n\subset \Sigma(a_i)$. The behavior of the stable map in neighborhoods of distinct nodal fibers is completely independent (see Figure 2 of \cite[page 386]{BL00}).

\subsection{Components of the moduli space on a hyperelliptic K3.}\label{subsec:FFcomponents}
Let $X_0 \rightarrow \P^{1} \times \P^{1}$ be a hyperelliptic K3 surface of genus $g$ branched along a curve $B$ of bidegree $(4,4)$ and $L_0=\pi^{\ast}\O(1,r)$, where $g=2r+1$. The morphism $B \rightarrow \P^{1}$ given by the second projection is of degree 4, and from Riemann-Hurwitz we know that a general (4, 4)-curve has 24 simple ramification points. These ramification points give exactly 24 fibers in the linear series $\abs{\O(0, 1)}$ which are tangent to $B$. Let $R_{i}$ ($i = 1, \ldots, 24$) denote the preimages of these fibers on the K3 surface, which are nodal rational curves. Let $T = \pi^{\ast}\O(1,0)$ be the class of an elliptic bisection and $F = \pi^{\ast}\O(0,1)$ a fiber class, so that
\[ F^{2} = 0, \quad F \cdot T = 2, \quad T^{2} = 0. \]

We now define some relevant substacks of $\overline{\M}_{1}(X_0,L_0)$. 
\begin{itemize}
    \item Let $\Omega \subset \overline{\M}_{1}(X_0, \pi^{\ast}\O(1, 1))$ be the locus of stable maps with smooth domain, whose image in $X_0$ contains exactly two nodes. By Remark~\ref{tangentnodes}, such curves are the pullbacks of irreducible $(1, 1)$-curves twice tangent to $B\subset \P^{1} \times \P^{1}$.
    \item Let $\F \subset \overline{\M}_{1}(X_{0}, L_0)$ be the locus of stable maps whose restriction to a genus 1 component of the domain lies in $\Omega$. The other components in the domain are rational curves that must map into the nodal fibers $R_{i}$. 
\end{itemize} 
Both substacks have proper closures, $\overline{\Omega}$ and $\overline{\F}$. We call $\overline{\F}$ the \textit{fixed fiber} component; it admits a natural forgetful map
\[ \tau \colon \overline{\F} \rightarrow \overline{\Omega}. \]
Occasionally we will also use $\tau$ to denote the map from $\F$ to $\Omega$. For a general $J\in \Omega$, the preimage $\F(J) \colonequals \tau^{-1}(J)$ is a proper divisor in $\overline{\M}_1(X_0,L_0)$ and it admits a decomposition via disjoint unions and products into spaces $\F_{s(a_{i})}(J)$ which are defined below. In general, $\F(J)$ will be proper but not smooth. In order to remedy this, we need to work with obstruction complexes rather than obstruction bundles in \S 4.3.

\begin{enumerate}[label=\roman*.]
    \item For $\vec{\bm{a}} = (a_{1}, a_{2}, \ldots , a_{48}) \in (\Z_{\geq 0})^{48}$, let $\F_{\vec{\bm{a}}}(J)$ be the moduli stack of stable, genus 1 maps to $X_0$ with image $J + \sum_{i=1}^{24} (a_{2i-1}+a_{2i}) R_{i}$. The two coefficients in front of each $R_{i}$ correspond to the fact that the elliptic bisection $J$ meets each $R_{i}$ in two points. 
    \item For any fixed nodal rational fiber $R$, let $\F_{a}(J)$ be the moduli stack of stable, genus 1 maps to $X_0$ with image $J + aR$.
    \item Let $E(a)$ be a genus 1 nodal curve consisting of a linear chain of $2a+1$ smooth rational components $E_{-a}, \ldots, E_{a}$ with an additional smooth elliptic component $E_{\ast}$ meeting $E_{0}$ (so that $E_{n} \cap E_{m} = \emptyset$ unless $\abs{n-m} = 1$ and $E_{\ast} \cap E_{n} = \emptyset$ unless $n = 0$). The moduli stack $\F_{s(a)}(J)$ is defined to be the moduli stack of genus 1 stable maps with $E(a)$ as the target in the class
\[ E_{\ast} + \sum_{n=-a}^{a} s_n(a) E_n.\]
As before, the admissible sequence has been extended by zero so that $n$ lies in $[-a,a]$.
\end{enumerate}

\begin{center}
\includegraphics[width=12cm]{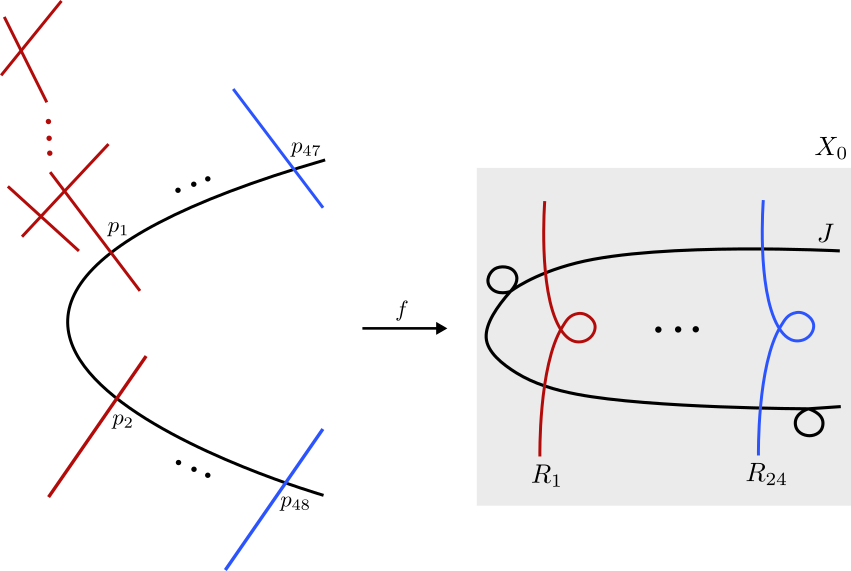}
\end{center}

Along the lines of Lemma \ref{lem:admissible sequence}, one can show that each $\F_{a}(J)$ is a disjoint union
\begin{equation*}
\coprod_{s(a)} \F_{s(a)}(J)
\end{equation*}
of spaces labeled by admissible sequences $s(a) = \{ s_n(a)\}$ with  $\abs{s(a)}=\sum_n s_n(a)=a$.

Then for a general $J\in \Omega$, $\F(J)$ has the following decomposition as in the Bryan-Leung case:

\begin{align}\label{1FF}
\F(J) &= \coprod_{a_1+\dots+a_{48}=r-1} \F_{\vec{\bm{a}}}(J) \\ \nonumber
&= \coprod_{a_1+\dots+a_{48}=r-1} \left( \prod_{i=1}^{48}\F_{a_i}(J) \right) \\ \nonumber
&= \coprod_{a_1+\dots+a_{48}=r-1} \left( \prod_{i=1}^{48} \coprod_{s(a_i)} \F_{s(a_i)}(J) \right)
\end{align}

\subsection{Identification of moduli stacks.} 

The curve $\Sigma(a)$ from the Bryan-Leung setting is almost the same as our curve $E(a)$; the only difference is that the smooth rational component $\Sigma_{\ast}\subseteq \Sigma(a)$ is replaced by a smooth elliptic component $E_{\ast}\subseteq E(a)$. As in their setup, we fix once and for all a map $E(a) \rightarrow X_{0}$, which sends $E_{\ast}$ to $J$ with degree 1 and each rational chain component $E_{n}$ maps to $R$ with degree 1. Other parts of the setups are identical.

\begin{prop}\label{isom moduli}
There is an isomorphism of moduli stacks
\begin{center}
\begin{tikzcd}
\F_{s(a)}(J) \arrow[r, "\sim"] & \smash{\overline{\M}}\vphantom{M}^{\BL}_{s(a)}.
\end{tikzcd}
\end{center}
\end{prop}

\begin{proof}

For any test scheme $T$, we have
\[ \smash{\overline{\M}}\vphantom{M}^{\BL}_{s(a)}(T) = \left\{
\begin{tikzcd}[column sep=small]
\cC_1 \arrow[rr, "u"] \arrow[rd, swap, "\text{flat}"] & & \Sigma(a) \times T \arrow[ld] \\
& T &
\end{tikzcd}
\right\} \]
where $\cC_1 \rightarrow T$ is a flat family of reducible nodal curves and the fiber over $t \in T$ is a stable map with target
\[ \Sigma_{\ast} + \sum_{n=-a}^{a} s_{n}(a) \Sigma_{n}. \]
Similarly, we have
\[ \F_{s(a)}(J)(T) = \left\{
\begin{tikzcd}[column sep=small]
\mathcal{C}_2 \arrow[rr, "v"] \arrow[rd, swap, "\text{flat}"] & & E(a) \times T \arrow[ld] \\
& T &
\end{tikzcd}
\right\} \]
where the fiber over $t \in T$ is a stable map with target
\[ E_{\ast} + \sum_{n=-a}^{a} s_n(a) E_n. \]

We now define a map of sets
\begin{equation}\label{eqn:morphism of stacks}
\F_{s(a)}(J)(T) \rightarrow \smash{\overline{\M}}\vphantom{M}^{\BL}_{s(a)}(T)
\end{equation} 
as follows. Given a diagram
\begin{center}
\begin{tikzcd}[column sep=small]
\mathcal{C}_2 \arrow[rr, "v"] \arrow[rd, swap, "\text{flat}"] & & E(a) \times T \arrow[ld] \\
& T &
\end{tikzcd}
\end{center}
as above, we may map the component $E_{\ast}$ to $\Sigma_{\ast}$, and there are isomorphisms $E_{n}\cong \Sigma_{n}$ which preserve the intersections between the components. There is a component $\cE$ of $\cC_2$ which maps isomorphically onto $E_{\ast} \times T$; this component will now be replaced with $\Sigma_{\ast} \times T$. Furthermore, $u^{-1}(\Sigma_{\ast} \times T)$ is a $\P^1$-bundle which admits a morphism of $\P^1$-bundles to $T \cong \P^{1} \times T$, so it is trivial. Since one can define an inverse morphism analogously, we conclude that the morphism (\ref{eqn:morphism of stacks}) is an isomorphism. The morphisms in the groupoids are defined similarly and it is easy to check they are isomorphisms.
\end{proof}

%
%

\section{Virtual intersection counts}\label{sec:VirtualIntersectionCounts}
In this section, we will compare the virtual fundamental classes $\smash{\overline{\M}}\vphantom{M}^{\BL}_{s(a)}$ and $\F_{s(a)}(J)$ to refine the isomorphism in Proposition~\ref{isom moduli}. We then use this to prove the main theorem in Section \ref{sec:proofmaintheorem}.

\subsection{Virtual Fundamental Classes.}\label{subsec:VFC}

We review the formalism of Behrend-Fantechi \cite{BF97} for associating a virtual fundamental class to a Deligne-Mumford stack $\fX$ endowed with a perfect obstruction theory $E^\bullet$. 

\begin{definition}
A \textit{perfect obstruction theory} on $\fX$ is a 2-term complex of vector bundles
\[ E^\bullet = [E^0\to E^1] \]
and a morphism $E^\bullet \to \L_\fX$ to the cotangent complex which is an isomorphism on $\mathcal H^0$ and surjective on $\mathcal H^1$. One can associate to $E^\bullet$ the vector bundle stack $\mathfrak{E}$ defined by the quotient 
\[ \mathfrak{E} = [E^1/E^0]. \]
\end{definition}
\begin{definition}
The \textit{intrinsic normal cone} of $\fX$ is defined as follows: first embed $\fX$ into an ambient smooth stack $\cA$, and then take the quotient
\[ \mathfrak N = [N_{\fX/\cA} / T_{\cA}]. \]
\end{definition}
\noindent
As a consequence of the definition of a perfect obstruction theory, there is a closed embedding $j \colon \mathfrak{N} \hookrightarrow \mathfrak{E}$.  We define $\text{vdim}(\fX) \colonequals \text{rk } E^0 - \text{rk } E^{1}$. Note that $\dim \mathfrak{N} = 0$ and $\dim \mathfrak{E} = \dim \fX -\text{vdim}(\fX)$. 
\begin{definition}
The virtual fundamental class $[\fX]^{\vir} \in CH_{\text{vdim}(\fX)}(\fX)$ is defined by 
\[ [\fX]^{\vir} \colonequals {\bf 0}^![\mathfrak N], \]
where $\textbf{0} \colon \fX \to \mathfrak E$ is the zero section.
\end{definition}
\noindent
Given an lci morphism of DM stacks $u:\fX'\to \fX$, there is a Gysin pullback $CH^*(\fX)\to CH^*(\fX')$ so the virtual class $[\fX]^{\vir}\in CH^*(\fX)$ can be restricted to $\fX'$. One can ask whether there is a restriction procedure at the level of the perfect obstruction theory. This is provided by the compatibility datum  in \cite{BF97}. 
Consider a Cartesian diagram
\begin{center}
\begin{tikzcd}
\fX' \arrow[r, "u"] \arrow[d, swap, "\tau'"] & \fX \arrow[d, "\tau"] \\
\B' \arrow[r, swap, "v"] & \B
\end{tikzcd}
\end{center}
where $v:\B' \to \B$ is an lci morphism.

\begin{definition}{\cite[Definition 5.8]{BF97}}\label{definition:compatible data}
Let $E\to \L_{\fX}$ and $E'\to \L_{\fX'}$ be perfect obstruction theories for $\fX$ and $\fX'$, respectively. A \textit{compatibility datum} relative to $v$ is a morphism of exact triangles in $D^b(\fX')$:
\begin{center}
\begin{tikzcd}
u^{\ast}E \arrow[r] \arrow[d] & E' \arrow[r] \arrow[d] & (\tau')^{\ast}\L_{\B'/\B} \arrow[r] \arrow[d] & u^{\ast}E[1] \arrow[d] \\
u^{\ast}\L_{\fX} \arrow[r] & \L_{\fX'} \arrow[r] & \L_{\fX'/\fX} \arrow[r] & u^{\ast}\L_{\fX}[1]
\end{tikzcd}
\end{center}
\end{definition}

\begin{prop}{\cite[Proposition 5.10]{BF97}}\label{prop:comparison of virtual classes}
Let $E$ and $E'$ be compatible perfect obstruction theories over $v$, as above, with their associated virtual classes. If $\B$ and $\B'$ are smooth, then
\[ v^{!}[\fX]^{\mathrm{vir}}=[\fX']^{\mathrm{vir}}.\]
\end{prop}

\subsection{Deformation invariance.}\label{subsec:definvariance}

To relate the genus one virtual class associated to the general K3 surface $(X,L)$ to that of the hyperelliptic specialization $(X_0,L_0)$, we consider relative versions of the notions in \S \ref{subsec:VFC}. Let $\epsilon:\mathcal{X} \to \Delta$  be a (polarized) family over an affine curve $\Delta$ which realizes the specialization, and $\mathcal L$ the polarizing line bundle on $\mathcal X$.
\begin{definition}
Let $\nu:\overline{\mathcal M}_1(\epsilon, \mathcal L)\to \Delta$ denote the moduli stack of genus 1 stable maps to the fibers of $\epsilon$.
\end{definition}
Note that $\nu$ is proper, but not flat. For example, $\nu^{-1}(0)$ has components of larger dimension as described in Section \ref{sec:hyperstable}. Nonetheless, $\overline{\mathcal M}_1(\epsilon, \mathcal L)$ carries a perfect relative obstruction theory $E^\bullet_{\epsilon}$ in the sense of \cite{BF97}. The associated virtual class satisfies the following compatibility property:
\begin{prop}\label{compatible}
For all $t\in \Delta$, let $\iota_t:\{t\} \to \Delta$ be the inclusion. The refined Gysin map takes
\[ \iota_t^! [\overline{\mathcal M}_1(\epsilon, \mathcal L)]^{\vir} = [\overline{\mathcal M}_1(\mathcal X_t, \mathcal L_t)]^{\vir}. \]
\end{prop}
\begin{proof}
This is a special case of \cite[Theorem 20(i)]{maulikpandthomas}. There they construct the perfect relative obstruction theory $E^\bullet_{\epsilon}$ for K3 families $\epsilon$ with potentially singular fibers, but $\mathcal L$ is still a primitive class.
\end{proof}
The class $[\overline{\mathcal M}_1(\epsilon, \mathcal L)]^{\vir}\in CH_2(\overline{\mathcal M}_1(\epsilon, \mathcal L),\mathbb Q)$ is represented by a linear combination of integral 2-cycles. Since $\Delta$ is a smooth affine curve, an integral 2-cycle $Z$ is flat over $\Delta$ if and only if it dominates $\Delta$. On the other hand, $Z$ is non-flat if and only if it is supported on a fiber of $\nu$. In that case, $\iota_t^![Z]=0$ for all $t\in \Delta$, so subtracting $[Z]$ from $[\overline{\mathcal M}_1(\epsilon, \mathcal L)]^{\vir}$ does not change the compatibility property of Prop. \ref{compatible}. Hence, we may assume without loss of generality that the relative virtual class is represented by a linear combination of cycles flat over $\Delta$.

Since $\overline{\M}_1^\circ(X,L)$ is reduced of the expected dimension, the virtual class is equal to the fundamental class on the general fiber. There is a unique way to extend 
\[\bigcup_{t\neq 0} \overline{\M}_1^\circ(\mathcal X_t,\mathcal L_t)\]
to a cycle flat and proper over $\Delta$, namely by taking the Zariski closure. Let $\overline{\M}_1^{\lim}(X_0,L_0)$ denote the central fiber of the closure. Recall from \S \ref{subsec:FFcomponents} that there is a forgetful map $\tau \colon \overline{\F} \rightarrow \overline{\Omega}$. The ghost components (see Definition~\ref{def:ghost}) of $\overline{\M}_1(X,L)$ specialize to cycles disjoint from $\tau^{-1}(J)=\F(J)$ for generic $J \in \Omega$. This is because the elliptic components in the domain of a stable map of $\F(J)$ must map to the elliptic curve $J$ itself. As a result, we have the following relation
\begin{equation}\label{interaction with ghost components}
    [\overline{\M}_1(X_0,L_0)]^{\vir}|_{\F(J)} = [\overline{\M}_1^{\lim}(X_0,L_0) ] |_{\F(J)},
\end{equation}
which will be sufficient to produce the desired bound.

\subsection{Virtual class comparison.} Let $(X_0,L_0)$ be a general hyperelliptic K3 surface of genus $g$ as in Set-up~\ref{setup:hyperelliptic K3}, and let $\Omega$ be the space defined in \S \ref{subsec:FFcomponents}. Let $J$ be a general element in $\Omega$. We will use the formalism above to compare two classes in $CH_0(\F(J))$. First, the virtual class on $\overline{\M}_1(S_0,L_0)$ in dimension 1 restricts to a dimension 0 class on the proper divisor $\F(J)$. Second, there is a perfect obstruction theory on each Bryan-Leung space $\smash{\overline{\M}}\vphantom{M}^{\BL}_{s(a)}$ with virtual dimension 0. Proposition~\ref{isom moduli} gives an isomorphism of stacks
\begin{equation}\label{stackisom}
    \F_{s(a)}(J) \cong \smash{\overline{\M}}\vphantom{M}^{\BL}_{s(a)}
\end{equation}
so each factor $\F_{s(a)}(J)$ carries a perfect obstruction theory. After taking products and disjoint unions, the resulting virtual class on $\F(J)$ is denoted $[\F(J)]^{\vir}$.

\begin{lem}\label{lemma:obstruction}
For a general element $J \in \Omega$, we have:
\[ [\overline{\M}_{1}(X_{0}, L_{0})]^{\vir} \big|_{\F(J)} = \left[ \F(J) \right]^{\vir}.\]
\end{lem}
\begin{proof}
To prove the comparison of virtual classes, we work at the level of obstruction theories, using a compatibility datum (see Proposition~\ref{prop:comparison of virtual classes}).

Fix a general element $J\in \Omega$ and recall that there is a morphism $\tau \colon \overline{\F} \rightarrow \overline{\Omega}$. Then in our context, $\B$ is an \'{e}tale neighborhood of $J$, $\B' = \{J\}$, $\fX \colonequals \tau^{-1}(\B) \subset \F$, and $\fX'=\F(J)$. These maps fit into the diagram below:
\begin{center}
\begin{tikzcd}
\F(J) \arrow[r, "u"] \arrow[d, swap, "\tau'"] & \fX \arrow[d, "\tau"] \\
\{J\} \arrow[r, swap, "v"] & \B
\end{tikzcd}
\end{center}
By choosing $J\in \Omega$ general and $\B\to \Omega$ to be an \'{e}tale neighborhood of $J$, we may assume that $\B$ is smooth, $\tau:\fX \to \B$ is flat, and $J$ is transverse to all of the nodal rational fibers in $X_{0}$\footnote{It suffices to produce a single example to check this open condition. Working on $\P^{1} \times \P^{1}$, we need an example of a $(4,4)$ curve $B$ and an irreducible $(1,1)$ curve $\Gamma$ such that $\Gamma$ is tangent to $B$ at exactly two points, which are disjoint from the 24 fibers of type $(0,1)$ coming from the simple ramification points of the second projection $B \rightarrow \P^{1}$. For instance, one can choose $\Gamma$ to be the graph of $y = 1/x$ and $B$ to be the union of two general cubics which are of type $(3,1)$ and $(1,3)$ such that each of them is tangent to $\Gamma$ at one point.}. Note that
\[\L_{\F(J)/\fX}\cong (\tau')^{\ast}\L_{\{J\}/\B} \cong T^*_{\{J\}}\B \otimes \O_{\F(J)}[1]\] 
is a shifted trivial bundle.

Let $E$ denote the obstruction theory on $\fX$ coming from $\overline{\M}_1(X_0,L_0)$, and let $E^{\BL}_{s(a)}$ denote the Bryan-Leung obstruction theory on $\mathcal F_{s(a)}(J)$ inherited through isomorphism \eqref{stackisom}. We will define a morphism
\begin{equation}\label{eqn:themorphism}
    T^*_{\{J\}}\B \otimes \O_{\F(J)} \rightarrow u^{\ast}E,
\end{equation}
and set
\begin{equation}\label{cone definition}
    E'\colonequals \mathrm{cone}\left( T^*_{\{J\}}\B \otimes \O_{\F(J)} \to u^{\ast}E\right),
\end{equation}  
so that the compatibility datum will follow from axiom TR3 of a triangulated category \cite[\href{https://stacks.math.columbia.edu/tag/0145}{Tag 0145}]{stacks-project}:
\begin{center}
\begin{tikzcd}
T^*_{\{J\}}\B \otimes \O_{\F(J)} \arrow[r] \arrow[d, "\sim"{sloped, above}] & u^{\ast}E\arrow[r] \arrow[d] & E' \arrow[r] \arrow[d, dashed] & T^*_{\{J\}}\B \otimes \O_{\F(J)}[1] \arrow[d, "\sim"{sloped, above}] \\
\L_{\F(J)/\fX}[-1] \arrow[r] & u^{\ast}\L_{\fX} \arrow[r] & \L_{\F(J)} \arrow[r] & \L_{\F(J)/\fX}.
\end{tikzcd}
\end{center}
In order to define (\ref{eqn:themorphism}), it is enough to define it for each admissible sequence $s(a)$. This is because the decomposition \eqref{1FF} of $\F(J)$ extends over the \'{e}tale neighborhood $\B$
\[\fX \cong \coprod_{a_1+\dots+a_{48}=r-1} \left( \prod_{i=1}^{48} \coprod_{s(a_i)} \F_{s(a)}(\B) \right), \]
where $\F_{s(a)}(J')$ is the fiber of $\F_{s(a)}(\B)$ over each $J'\in\B$ (compare with \cite[line (7) of pg. 395]{BL00}).

Recall from Proposition~\ref{isom moduli} that the universal family $\cC_1$ over the Bryan-Leung space $\smash{\overline{\M}}\vphantom{M}^{\BL}_{s(a)}$ has two components $\cC_{1} = \cS \cup \cD_1$, where $\cS$ is the universal smooth rational component and $\cD_1$ is the universal rational tree component. Likewise, the universal family for $\F_{s(a)}(\B)$ breaks up into a union of two components: $\cC_{2} = \cE \cup \cD_2$, where $\cE$ is the universal elliptic component and $\cD_2$ is the universal rational tree component.

The edge cutting axiom (Axiom III of \cite{Behrend97GW}) for Gromov-Witten obstruction theories gives an exact triangle relating the Gromov-Witten obstruction theory for the separate components, the obstruction theory for their union, and the cotangent bundle of the target variety pulled back through the evaluation map at the marked point. Consider the following two Cartesian square diagrams, the first for the Bryan-Leung moduli space and the second our moduli space:
\begin{center}
\begin{tikzcd}
\smash{\overline{\M}}\vphantom{M}^{\BL}_{s(a)} \arrow[r, "{(\tau_{\cS},\tau_{\cD_1})}"] \arrow[d, swap, "y"] & {\cS\times \cD_1} \arrow[r] \arrow[d] & \mathrm{pt}; \\
Y\arrow[r, swap, "\delta"] & Y\times Y &
\end{tikzcd}
\end{center}

\begin{center}
\begin{tikzcd}
\F_{s(a)}(J) \arrow[r,"u"]  & \F_{s(a)}(\B) \arrow[r, swap, "{(\tau_{\cE}, \tau_{\cD_{2}})}"] \arrow[rr, bend left, "\tau"] \arrow[d, swap, "x_{0}"] & {\cE \times \cD_{2}} \arrow[r] \arrow[d] & \B ,\\
{} & X_0 \arrow[r, swap, "\delta"] & X_0\times X_0 & 
\end{tikzcd}
\end{center}
where in the first diagram $\tau_{\cS}$ and $\tau_{\cD_1}$ are the maps that remember the attaching point of the smooth rational component with the rational tree component, and in the second diagram $\tau_{\cE}$ and $\tau_{\cD_2}$ are the maps that remember the attaching point of the smooth elliptic component with the rational tree component. In a slight abuse of notation, we will use $u$ to refer to both the map $\F(J) \rightarrow \fX$ and the maps $\F_{s(a)}(J) \rightarrow \F_{s(a)}(\B)$.

From \cite[pg. 608]{Behrend97GW}, there are exact triangles
\begin{equation}\label{obstruction theory for EBL}
    y^*\L_{Y} \to \tau_{\cS}^{\ast} E_{\cS} \oplus \tau_{\cD_1}^{\ast} E_{\cD_1} \to E^{\BL}_{s(a)};
\end{equation}
\begin{equation}\label{obstruction theory for FJ}
    x_0^*\L_{X_0} \to \tau_{\cE}^{\ast}E_{\cE} \oplus \tau_{\cD_2}^{\ast} E_{\cD_2} \to E_{s(a)}. 
\end{equation}
We may apply $u^{\ast}$ to the latter triangle to obtain:
\begin{equation}\label{pullbacks of equations from EL}
    u^{\ast}x_0^{\ast} \L_{X_0} \to u^{\ast} \tau_{\cE}^{\ast} E_{\cE} \oplus u^{\ast}\tau_{\cD_{2}}^{\ast} E_{\cD_2} \to u^{\ast} E_{s(a)}.
\end{equation}


\noindent \textbf{Claim:} There is an exact triangle
\begin{equation}\label{Tate curve exact triangle}
    u^{\ast}\tau^{\ast} \L_\B \to u^\ast \tau_{\cE}^\ast E_{\cE} \oplus u^{\ast} \tau_{\cD_2}^\ast E_{\cD_2} \to \tau_{\cS}^{\ast} E_{\cS} \oplus \tau_{\cD_1}^{\ast} E_{\cD_1}
\end{equation}
where the third object is pulled back through the isomorphism of Proposition \ref{isom moduli}.

\begin{proof}[Proof of Claim.]\renewcommand{\qedsymbol}{}
The last two objects in the purported exact triangle are defined in terms of Kontsevich moduli spaces with different targets: $X_0$ and $Y$, respectively. Both are K3 surfaces with genus one fibrations, but $X_0$ is hyperelliptic with a bisection $J$ and $Y$ is unigonal with section $S$. Both fibrations have 24 nodal fibers, and furthermore a formal neighborhood of any nodal fiber (on either surface) is isomorphic to the Tate curve over $\mathbb{C}\llbracket t \rrbracket$, viewed as families of 1-marked curves of genus 1. The markings are given by sections of the Tate curve: a neighborhood of $y\in S$, and a neighborhood of $x_0\in J$. Fix an isomorphism $\Phi$ over $\mathbb{C} \llbracket t \rrbracket$ between the formal neighborhoods $\mathcal U_{N_1}\subset Y$ and $\mathcal U_{N_2}\subset X_0$ respecting the markings (any two sections of the Tate curve differ by translation).

The obstruction theories $\tau_{\cD_1}^\ast E_{\cD_1}$ and $u^{\ast} \tau_{\cD_2}^\ast E_{\cD_2}$ are canonically isomorphic; they are the Gromov-Witten obstruction theories for genus 0, 1-marked maps to a neighborhood of the nodal fiber. Hence, it suffices to show the exactness of 
\[u^{\ast}\tau^{\ast} \L_\B \to u^\ast \tau_{\cE}^\ast E_{\cE} \to \tau_{\cS}^{\ast} E_{\cS}. \]
Note that each term here can be represented by a perfect complex supported in degree 0, so it suffices to show that there is a short exact sequence of vector bundles:
\[ 0 \to T^\ast_{\{J\}}\B \otimes \O_{\F(J)} \to \Def(J,x_0)^\ast \otimes \O_{\F(J)} \to \Def(S,y_0)^\ast \otimes \O_{\F(J)} \to 0, \]
which can be rewritten (dually) as
\[ 0 \to \Def(S,y_0) \otimes \O_{\F(J)} \to \Def(J,x_0)\otimes \O_{\F(J)} \to T_{\{J\}}\B \otimes \O_{\F(J)} \to 0 .\]
The last map is given by the differential of the universal family $\cE \to \B$. The kernel of this map is the tangent space to $J$ at $x_0$, which is identified with the tangent space to $S$ at $y$ via the (restricted) isomorphism $\Phi$:
\begin{center}
\begin{tikzcd}
T_y Y \arrow[r, "d\Phi"]  & T_{x_0} X_0   \\
T_y S \arrow[r, swap, "d\Phi|_{T_{y}S}"] \arrow[u] & T_{x_0} J.\arrow[u]
\end{tikzcd}
\end{center}
This completes the proof of the claim.
\end{proof}


Consider the first morphism from (\ref{Tate curve exact triangle}):
\[ T^\ast_{\{J\}}\B \otimes \O_{\F(J)}\overset{p}\longrightarrow u^{\ast}\tau_{\cE}^{\ast}E_{\cE} \oplus u^{\ast}\tau_{\cD_2}^{\ast}E_{\cD_2}, \]
and the second morphism from (\ref{pullbacks of equations from EL}):
\[u^{\ast}\tau_{\cE}^{\ast}E_{\cE} \oplus u^{\ast}\tau_{\cD_2}^{\ast}E_{\cD_2} \overset{q}\longrightarrow u^\ast E_{s(a)}.\]
We now define the desired morphism (\ref{eqn:themorphism}) for the compatibility datum by taking an external direct sum over $s(a)$ for the product decomposition (\ref{1FF}) of the composed morphisms:
\[ q\circ p\colon  T^*_{\{J\}}\B \otimes \O_{\F(J)} \overset{p}\longrightarrow u^{\ast}\tau_{\cE}^{\ast}E_{\cE} \oplus u^{\ast}\tau_{\cD_2}^{\ast}E_{\cD_2} \overset{q}\longrightarrow u^{\ast}E_{s(a)}.\]
Using the octahedral axiom TR4 of a triangulated category \cite[\href{https://stacks.math.columbia.edu/tag/0145}{Tag 0145}]{stacks-project} and the exact triangles above, we have the following diagram:

\begin{center}
\begin{tikzpicture}[scale=3]
  \def\openingangle{45}
  \node (Z') at (0, 0) {$\tau_{\cS}^{\ast}E_{\cS} \oplus \tau_{\cD_1}^{\ast} E_{\cD_1}$};
  \node (Y) at ($(Z') +(270-\openingangle:1)$) {$u^{\ast}\tau_{\cE}^{\ast}E_{\cE} \oplus u^{\ast}\tau_{\cD_2}^{\ast}E_{\cD_2}$};
  \node (Y') at ($(Z') +(270+\openingangle:1)$) {$E'_{s(a)}$};
  \node (X) at ($(Y) +(270-\openingangle:1)$) {$T^*_{\{J\}}\B \otimes \O_{\F_{s(a)}(J)}$};
  \node (X') at ($(Y') +(270+\openingangle:1)$) {$W[1]$};
  \draw[white, name path=L1] (X) -- (Y');
  \draw[white, name path=L2] (X') -- (Y);
  \path[name intersections={of=L1 and L2, by=pZ}];
  \node (Z) at (pZ) {$u^{\ast}E_{s(a)}$};

  \draw[-latex] (X) -- node[above left]{$p$} (Y);
  \draw[-latex] (X) -- node[below right]{$q \circ p$} (Z);
  \draw[-latex] (Y) -- node[above left]{\eqref{Tate curve exact triangle}} (Z');
  \draw[-latex] (Y) -- node[below left]{$q$} (Z);
  \draw[-latex] (Z) -- (Y');
  \draw[-latex] (Z) -- node[below left]{\eqref{pullbacks of equations from EL}} (X');
  \draw[-latex, dashed] (Y') -- (X');
  \draw[-latex, dashed] (Z') -- (Y');
\end{tikzpicture}
\end{center}
Here, $W$ is given by
\[ W\colonequals u^{\ast}x_{0}^{\ast}\L_{X_{0}} \cong T^*_{x_0}X_0 \overset{d\Phi^\ast} \longrightarrow T^\ast_y Y \cong y^\ast \L_Y,\]
and $E'_{s(a)} \colonequals \mathrm{cone} \left( T^{\ast}_{J}\B \otimes \O_{\F_{s(a)}(J)} \rightarrow u^{\ast}E_{s(a)} \right)$. Note that every linear triple is an exact triangle, and the exactness of the dotted triangle is the consequence of axiom TR4. In particular, we see that
\begin{equation}\label{identification of obstruction theory}
    E'_{s(a)} \cong \mathrm{cone}(W\to \tau^{\ast}_{\cS}E_{\mathcal{S}}\oplus \tau^{\ast}_{\cD_1}E_{\cD_1}) \cong E^{\BL}_{s(a)},
\end{equation} 
where the last isomorphism comes from \eqref{obstruction theory for EBL}.
Denote $[\F(J)]^{\mathrm{vir}}_{E'}$ to be the virtual class on $\F(J)$ induced by the perfect obstruction theory $E'$, the external direct sum of the obstruction theories $E'_{s(a)}$. By Proposition \ref{prop:comparison of virtual classes} and the isomorphism (\ref{identification of obstruction theory}), we conclude that
\[ [\overline{\M}_{1}(X_{0}, L_{0})]^{\vir} \big|_{\F(J)} = \left[ \F(J) \right]^{\vir}_{E'}=\left[ \F(J) \right]^{\vir}. \qedhere \]
\end{proof}

\section{Proof of Theorem~\ref{thm:Main}.}\label{sec:proofmaintheorem}

By Corollary~\ref{severikontsevich}, for a very general K3 surface $(X,L)$ we know that the compactified Severi curve $\overline{V}^{L, g-1}$ is birational to $\overline{\M}_1^\circ(X,L)$ component by component. As mentioned above, $\overline{\M}_1^\circ(X,L)$ is reduced of the expected dimension, so its virtual class is equal to the fundamental class. Let us specialize from $(X, L)$ to a very general hyperelliptic K3 surface $(X_{0}, L_{0})$ as in Set-up~\ref{setup:hyperelliptic K3}, with $g = 2r + 1$. We have a family of Kontsevich moduli spaces, each equipped with a virtual class. By definition, $\overline{\M}_1^\circ(X,L)$ has flat limit
\[ \overline{\M}^{\lim}_1(X_0,L_0) \subset  \overline{\M}_1(X_0,L_0). \]
Proposition~\ref{FFcounting} below computes the intersection number
\[ [\overline{\M}^{\lim}_1(X_0,L_0) ]\cdot \F(J) > 0. \]
This means that some components of $\overline{\M}^{\lim}_1(X_0,L_0)$ are contained in $\overline{\F}$, and they dominate $\overline{\Omega}$. In fact, \cite[Lemma 5.8]{BL00} implies that \textit{all} components of $\overline{\M}^{\lim}_1(X_0,L_0)$ over $\overline{\Omega}$ are reduced because the contribution of each connected component of $\F(J)$ to the virtual count is either 1 or 0; see Proposition~\ref{FFcounting} below. After restricting to a component if necessary, we may assume that $\overline{\Omega}$ is irreducible. Let $\bigcup_{i} V_i$ be the union of all components of $\overline{\M}^{\lim}_1(X_0,L_0)$ which dominate $\overline{\Omega}$. The bound for the total geometric genus of $\overline{V}^{L,g-1}$ follows from:
\begin{equation}\label{eq:TotalGenusInequality}\def\arraystretch{1.5}
\begin{array}{rll}
\genus \left(\overline{V}^{L, g-1}\right) &= \genus \left( \overline{\M}^\circ_{1}(X, L) \right) & \\
&\geq \genus \left( \bigcup_{i} V_i \right) & [\text{by Lemma~\ref{lemma:semicontinuityGeomGenus}}] \\
&> \sum_i \left( \genus (V_i)-1 \right) & \\
&\geq \left( [\overline{\M}^{\lim}_1(X_0,L_0) ]\cdot \F(J) \right) \left(\genus (\overline{\Omega}) -1\right), & [\text{by the Riemann-Hurwitz formula}]
\end{array}
\end{equation}
for some constant $C >0$. Here, genus means total geometric genus.

It now suffices to show that $\genus(\overline{\Omega}) > 1$ and give an asymptotic for the intersection number.

\begin{lem}\label{lemma:G9}
For a very general hyperelliptic $(X_0,L_0)$, the geometric genus of each irreducible component of $\overline{\Omega}$ is at least 9. 
\end{lem}

\begin{proof}
Let $B$ be a very general smooth $(4, 4)$ curve on $\P^{1} \times \P^{1}$. The curve
\[ Z_{2} \colonequals \{ A \in \Sym^{2}(B) : \exists D \in \abs{\O(1, 1)} \text{ such that } D \text{ is tangent to $B$ at } A \} \]
is birational to $\overline{\Omega}$, so it suffices to consider $Z_2$.

We observe that any smooth curve $B \subset \P^{1} \times \P^{1}$ of bi-degree $(a, b)$ (where $a \leq b$) satisfies $\gon(B) = a$. To see this, note that $\gon(B) \leq a$ from the projection(s) to $\P^{1}$. Given a smooth projective curve $C$, recall from \cite[Lemma 1.3]{BDELU17} that if $K_C$ is $p$-very ample, then $\gon(C)\geq p+2$. Since $K_B \cong \O(a-2,b-2)|_B$ is $(a-2)$-very ample, it follows that $\gon(B)=a$. In particular, our $(4,4)$ curve $B$ has gonality 4 and is not hyperelliptic. 

Since the genus of $B$ is 9 and $B$ is not hyperelliptic, the Abel-Jacobi map $\Sym^{2}(B) \rightarrow \Jac(B)$ is an embedding and therefore $Z_{2}$ embeds into $\Jac(B)$. It follows from \cite[Corollary 1.2]{CG92} that $\Jac(B)$ is simple because $H^{1, 0}(\P^{1} \times \P^{1}) = 0$ and $B$ is a very general element of a linear series. The Jacobian of the normalization of each component of $Z_{2}$ has a nonconstant map to $\Jac(B)$, so the geometric genus of each component is at least 9.
\end{proof}

\begin{prop}\label{FFcounting}
For a general element $J \in \Omega$, we have
\[ [ \overline{\M}^{\lim}_1(X_0,L_0) ]\cdot \F(J) = \left[ \prod_{m=1}^{\infty} (1-q^{m})^{-48} \right]_{q^{r-1}}= O(e^{C \sqrt{r}}) \]
for some constant $C > 0$.
\end{prop}
\begin{proof}
By equation~\eqref{interaction with ghost components} and Lemma~\ref{lemma:obstruction}, we have the following equalities
\[ [\overline{\M}^{\lim}_1(X_0,L_0)]\cdot \F(J) = [ \overline{\M}_1(X_0,L_0) ]^\vir\cdot \F(J)=[\F(J)]^{\vir}. \]
Recall that the virtual class has a decomposition which is compatible with the decomposition of $\F(J)$:
\[ \F(J) =\coprod_{a_1+\dots+a_{48}=r-1} \left( \prod_{i=1}^{48} \coprod_{s(a_i)} \F_{s(a_i)}(J) \right). \]
Each $\F_{s(a_i)}(J)$ carries an obstruction theory via the isomorphism in Proposition~\ref{isom moduli}. It suffices to calculate the degree of $[\F_{s(a)}(J)]^{\vir}$, which is equal to the degree of $[\smash{\overline{\M}}\vphantom{M}^{\BL}_{s(a)}]^{\vir}$.

We call an admissible sequence $\{ s_{n} \}$ \textit{1-admissible} if $s_{\pm n \pm 1}$ is either $s_{\pm n}$ or $s_{\pm n} - 1$ for all non-negative $n$.
By \cite[Lemma 5.8]{BL00}, $[\smash{\overline{\M}}\vphantom{M}^{\BL}_{s(a)}]^{\vir} = 1$ if $s(a)$ is a 1-admissible sequence and $[\smash{\overline{\M}}\vphantom{M}^{\BL}_{s(a)}]^{\vir} = 0$ otherwise.
By \cite[Lemma 5.9]{BL00}, the number of 1-admissible sequences $s$ with $\abs{s} = a$ is equal to the number of partitions of $a$, by taking the lengths of diagonals in the Young tableaux.
Therefore
\begin{align*}
[\overline{\M}^{\lim}_1(X_0,L_0)]\cdot \F(J) &= \sum_{a_1+\cdots+a_{48}=r-1} \left(\prod_{i=1}^{48} p(a_{i}) \right)\\
&= \left[ \prod_{m=1}^{\infty} (1-q^{m})^{-48} \right]_{q^{r-1}} \\
&= O(e^{C \sqrt{r}})
\end{align*}
for some constant $C > 0$. The last equality follows from the Hardy-Ramanujan asymptotic for the partition function \cite{HR1917}.
\end{proof}

Finally, we may combine equation~\eqref{eq:TotalGenusInequality} with Lemma~\ref{lemma:G9} and Proposition~\ref{FFcounting} to complete the proof of Theorem~\ref{thm:Main}.

%
%

\section{Open problems}

In this section, we collect some natural remaining questions. For $\P^{2}$, irreducibility of the Severi variety of degree $d$ and geometric genus $g$ curves was proved by Harris \cite{Harris86}. As mentioned in the introduction, the Severi variety $\overline{V}^{L, \delta}$ of a very general primitively polarized K3 surface $(X, L)$ with $c_{1}(L)^{2} = 2g-2$ is irreducible when $\delta \leq g-4$ (see \cite{BLC21}). On the other extreme, when $\delta = g-1$, one can ask:

\begin{problem}\label{problem}
Is the compactified Severi curve $\overline{V}^{L, g-1}$ irreducible?
\end{problem}
In a similar vein, it may be easier to give an upper bound for the maximum number of irreducible components of $\overline{V}^{L, g-1}$. Comparing this with the bound in Theorem~\ref{thm:Main} would give an approach to the less ambitious:

\begin{problem}
As $g\to\infty$, does $\overline{V}^{L, g-1}$ have an irreducible component of arbitrarily high geometric genus?
\end{problem}

Lemma~\ref{lemma:G9} implies that some irreducible components of $\overline{V}^{L, g-1}$ have geometric genus at least 9 (adapting the arguments in \S\ref{sec:proofmaintheorem}). Ultimately, one hopes for an explicit formula in the spirit of the Yau-Zaslow formula \cite{YZ96}.
\begin{problem}
Find an explicit formula for the total geometric genus of $\overline{V}^{L, g-1}$ on a very general K3 surface.
\end{problem}

More speculatively, we expect a similar result in higher dimensions:
\begin{problem}
For $\delta$ close to $g$, does the Severi variety $\overline{V}^{L,\delta}$ exhibit hyperbolic behavior in the sense that it does not contain rational or elliptic curves?
\end{problem}

\newpage

%
%

\bibliographystyle{siam}
\bibliography{Biblio}

\begin{thebibliography}{10}

\bibitem{BDELU17}
{\sc F.~Bastianelli, P.~{De Poi}, L.~Ein, R.~Lazarsfeld, and B.~Ullery}, {\em
  Measures of irrationality for hypersurfaces of large degree}, Compos. Math.,
  153 (2017), pp.~2368--2393.

\bibitem{Behrend97GW}
{\sc K.~Behrend}, {\em Gromov-witten invariants in algebraic geometry},
  Inventiones mathematicae, 127 (1997), pp.~601--617.

\bibitem{BF97}
{\sc K.~Behrend and B.~Fantechi}, {\em The intrinsic normal cone}, Invent.
  Math., 128 (1997), pp.~45--88.

\bibitem{BLC21}
{\sc A.~Bruno and M.~Lelli-Chiesa}, {\em Irreducibility of {S}everi varieties
  on {K}3 surfaces}, 2021.

\bibitem{BL00}
{\sc J.~Bryan and N.~Leung}, {\em The enumerative geometry of {K}3 surfaces and
  modular forms}, J. Amer. Math. Soc., 13 (2000), pp.~371--410.

\bibitem{Chen02}
{\sc X.~Chen}, {\em A simple proof that rational curves on {K}3 are nodal},
  Math. Ann., 324 (2002), pp.~71--104.

\bibitem{Chen19}
\leavevmode\vrule height 2pt depth -1.6pt width 23pt, {\em Nodal curves on {K}3
  surfaces}, New York J. Math., 25 (2019), pp.~168--173.

\bibitem{CS97}
{\sc L.~Chiantini and E.~Sernesi}, {\em Nodal curves on surfaces of general
  type}, Math. Ann., 307 (1997), pp.~41--56.

\bibitem{CD19}
{\sc C.~Ciliberto and T.~Dedieu}, {\em On the irreducibility of {S}everi
  varieties on {K}3 surfaces}, Proc. Amer. Math. Soc., 147 (2019),
  pp.~4233--4244.

\bibitem{CG92}
{\sc C.~Ciliberto and G.~van~der Geer}, {\em On the {J}acobian of a hyperplane
  section of a surface}, in Classification of irregular varieties, Springer,
  1992, pp.~33--40.

\bibitem{EL18}
{\sc L.~Ein and R.~Lazarsfeld}, {\em The {K}onno invariant of some algebraic
  varieties}, Eur. J. Math.,  (2018), pp.~1--10.

\bibitem{FP96}
{\sc W.~Fulton and R.~Pandharipande}, {\em Notes on stable maps and quantum
  cohomology}, arXiv:alg-geom/9608011,  (1996).

\bibitem{HR1917}
{\sc G.~H. Hardy and S.~Ramanujan}, {\em Asymptotic formulae for the
  distribution of integers of various types}, Proceedings of the London
  Mathematical Society, 2 (1917), pp.~112--132.

\bibitem{Harris86}
{\sc J.~Harris}, {\em On the {S}everi problem}, Invent. Math., 84 (1986),
  pp.~445--461.

\bibitem{Knudsen}
{\sc F.~F. Knudsen}, {\em The projectivity of the moduli space of stable
  curves. {III}. {T}he line bundles on {$M_{g,n}$}, and a proof of the
  projectivity of {$\overline M_{g,n}$} in characteristic {$0$}}, Math. Scand.,
  52 (1983), pp.~200--212.

\bibitem{maulikpandthomas}
{\sc D.~Maulik, R.~Pandharipande, and R.~P. Thomas}, {\em Curves on {$K3$}
  surfaces and modular forms}, J. Topol., 3 (2010), pp.~937--996.
\newblock With an appendix by A. Pixton.

\bibitem{Reid76}
{\sc M.~Reid}, {\em Hyperelliptic linear systems on a {K}3 surface}, J. Lond.
  Math. Soc., 2 (1976), pp.~427--437.

\bibitem{Saint-Donat74}
{\sc B.~Saint-Donat}, {\em Projective models of {K}-3 surfaces}, Amer. J.
  Math., 96 (1974), pp.~602--639.

\bibitem{stacks-project}
{\sc T.~{Stacks Project Authors}}, {\em \textit{Stacks Project}}.
\newblock \url{https://stacks.math.columbia.edu}, 2018.

\bibitem{YZ96}
{\sc S.-T. Yau and E.~Zaslow}, {\em B{PS} states, string duality, and nodal
  curves on {$K3$}}, Nuclear Phys. B, 471 (1996), pp.~503--512.

\end{thebibliography}

\footnotesize{
\textsc{Department of Mathematics, Stony Brook University, Stony Brook, New York 11794} \\
\indent \textit{E-mail address:} \href{mailto:nathan.chen@stonybrook.edu}{nathan.chen@stonybrook.edu}

\vspace{\baselineskip}

\textsc{Simons Center for Geometry and Physics and Department of Mathematics, Stony Brook University, Stony Brook, New York 11794} \\
\indent \textit{E-mail address:} \href{mailto:fgreer@scgp.stonybrook.edu}{fgreer@scgp.stonybrook.edu}

\vspace{\baselineskip}

\textsc{Department of Mathematics, Stony Brook University, Stony Brook, New York 11794} \\
\indent \textit{E-mail address:} \href{mailto:ruijie.yang@stonybrook.edu}{ruijie.yang@stonybrook.edu}

}

\end{document}